\documentclass[11pt]{article}
\usepackage{graphicx,verbatim}
\usepackage{lineno}
\usepackage{lipsum,pdflscape}
\usepackage{amsmath,amssymb}
\usepackage{xcolor,colortbl}
\usepackage{xspace}
\usepackage{color}
\usepackage{epsfig}
\usepackage[algosection,lined,ruled,linesnumbered]{algorithm2e}
\usepackage{subfigure}

\graphicspath{{./Fig/}}
\newtheorem{theorem}{Theorem}[section]
\newtheorem{proposition}[theorem]{Proposition}

\newtheorem{corollary}[theorem]{Corollary}

\newcommand{\proof}{\noindent{\bf Proof. }}
\newcommand{\qed}{\hfill $\square$\medskip}

\begin{document}

\title{More on graph pebbling number}

\author{
Saeid Alikhani$^{}$\footnote{Corresponding author}
\and
Fatemeh Aghaei
}

\date{\today}

\maketitle

\begin{center}
Department of Mathematical Sciences, Yazd University, 89195-741, Yazd, Iran\\
{\tt alikhani@yazd.ac.ir, aghaeefatemeh29@gmail.com}
\end{center}

%%%%%%%%%%%%%%ABSTRACT%%%%%%%%%%%%%%%%%%%%%%%%%%%%%%%%%%%%%%%%%%%%%%%%%%%%%%%%%%%%

\begin{abstract}
 Let $G=(V,E)$ be a simple graph. 
 A function $\phi:V\rightarrow \mathbb{N}\cup \{0\}$ is called a configuration of pebbles on the vertices of $G$  and the quantity $\sum_{u\in V}\phi(u)$
 is called the size of $\phi$ which is  just the total number of pebbles assigned to vertices. A pebbling step from a vertex $u$ to one of its
 neighbors $v$  reduces $\phi(u)$ by two and increases $\phi(v)$  by one. 
 Given a specified target vertex $r$ we
 say that $\phi$ is $t$-fold $r$-solvable, if some sequence of pebbling
 steps places at least $t$  pebbles on $r$. Conversely, if no such steps
 exist, then $\phi$ is $r$-unsolvable. The minimum positive integer $m$ such that
 every configuration of size $m$ on the vertices of $G$ is $t$-fold
 $r$-solvable is denoted by $\pi_t(G,r)$. The $t$-fold pebbling number of $G$ is defined to be
 $\pi_t(G)= max_{r\in V(G)}\pi_t(G,r)$.  When $t=1$,  we simply write
 $\pi(G)$, which is the pebbling number of $G$. In this note, we study the pebbling number for some specific graphs.  Also we investigate the pebbling number of corona and  neighbourhood corona of two graphs.

\end{abstract}


\medskip
\noindent{\bf AMS Subj.\ Class.}:   05C76.

%%%%%%%%%%%%%%%%%%%%%%%%%%%%%%%%%%%%%%%%%%%%%%%%%%%%%%%%%%%%%%%%%%%%%%%%%%%%%%%%%
%%%%%%%%%%%%%%%%%%%%%%%%%%%%%%%%%%%%%%%%%%%%%%%%%%%%%%%%%%%%%%%%%%%%%%%%%%%%%%%%%
\section{Introduction and definitions}
 Let $G=(V,E)$ be a simple graph of order $n$. 
 A function $\phi:V\rightarrow \mathbb{N}\cup \{0\}$ is called a configuration of pebbles on the vertices of $G$  and the quantity $\sum_{u\in V}\phi(u)$
 is called the size of $\phi$ which is  just the total number of pebbles assigned to vertices. A pebbling step from a vertex $u$ to one of its
 neighbors $v$  reduces $\phi(u)$ by two and increases $\phi(v)$  by one. 
 Given a specified target vertex $r$, we
 say that $\phi$ is $t$-fold $r$-solvable, if some sequence of pebbling
 steps places $t$ (or more)  pebbles on $r$. Conversely, if no such steps
 exist, then $\phi$ is $r$-unsolvable. 
 Furthermore,
 a configuration is solvable whenever it is $r$-solvable for every vertex $r$. 
 The minimum positive integer $m$ such that
 every configuration of size $m$ on the vertices of $G$ is $t$-fold
 $r$-solvable is denoted by $\pi_t(G,r)$. The $t$-fold pebbling number of $G$ is defined to be
 $\pi_t(G)= max_{r\in V(G)}\pi_t(G,r)$.  When $t=1$,  we simply write
 $\pi(G)$, which is the pebbling number of $G$. It is obvious
 that $\pi(G)\leq \pi_2(G) \leq 2\pi(G)$.

The configuration with a single pebble on every vertex
except the target shows that $\pi(G)\geq n$, while the configuration with
$2^{ecc(r)}-1$  pebbles on the farthest vertex from $r$, and no
pebbles elsewhere, shows that $\pi(G)\geq 2 {\rm diam}(G)$ when $r$ is
chosen to have $ecc(r)=diam(G)$.
As usual let $Q^d$ be the $d$-dimensional hypercube  as the graph on all binary $d$-tuples,
pairs of which that differ in exactly one coordinate are
joined by an edge. Chung \cite{Chung} proved that $\pi(Q^d)=2^d$. Graphs $G$ that, like
$Q^d$  have $\pi(G)=|V(G)|$  have come to be known as Class $0$.
The terminology comes from a lovely theorem of Pachter,
Snevily, and Voxman \cite{Pachter}, which states that if $diam(G)=2$, 
then $\pi(G)\leqslant n+1$. 
Therefore there are two classes of diameter two graphs, Class $0$ and
 Class $1$. The
Class $0$ graphs are $2$-connected. 

The cover pebbling number $ \gamma(G) $ to be the minimum number of pebbles needed to place at least a pebble on every vertex of the graph using a sequence of pebbling moves, regardless of the initial configuration. In \cite{Betsy}, the cover pebbling number for several classes of graphs, including complete graphs, paths, fuses (a fuse is a path with leaves attached at one end), and more generally, trees, is established. They also described the structure of the largest non-coverable configuration on a tree. It is not hard to see that $\gamma(K_{n})=2n-1$ and 
$ \gamma(P_{n})=2^{n}-1.$ Note that $\pi(P_n)=2^{n-1}$,  $\pi(K_n)=n$ and the ratio 
$\frac{\gamma(G)}{\pi(G)}$ is unbounded, even within the class of trees (\cite{Betsy}). 

\medskip
In this note,  we study the pebbling number for some specific graphs.  Also we investigate the pebbling number of corona and  neighbourhood corona of two graphs.

\section{Main results} 
 Let to start with  the cover pebbling number of cycle graph and $Q^n$.   We have a lower bound for $\gamma(C_n)$ by putting all the pebbles on one vertex and reaching the other vertices from that vertex. Sjostrand proved the exact value of cover pebbling number of cycle graph and $Q^n$ as follows:
 
 \begin{theorem}{\rm\cite{sjostrand}}
\begin{itemize}
\item[(i)]
$\gamma(C_{2k})=3(2^{k}-1)$,
\item[(ii)]
$\gamma(C_{2k-1})= 2^{k+1}-3.$
\item[(iii)]
 $ \gamma(Q^{n})=3^{n}.$ 

\end{itemize}
\end{theorem}

%result 2) Lower bound for the cover pebbling number of Cube graph is $ \gamma(Q^{3})\geqslant27 $.

Let $F_n$ be the friendship graph with $2n+1$  vertices and $3n$ edges, formed by the join of
$K_1$ with $nK_2$. 

\begin{theorem} 
	\begin{enumerate} 
		\item[(i)]
		The friendship graph $F_n$ is Class $1$. 
		\item[(ii)]
		$\pi_{2}(K_{n})=n+2.$
		\end{enumerate} 
	 \end{theorem} 
	 \proof 
	 		\begin{enumerate} 
	 			\item[(i)]
	 	Note that $ \pi(F_{n})\leqslant 2n+2$. The configuration $C$ with  $C(r)=C(u)=0$, $C(v)=3$ and $ C(w)=1 $ for other vertices $w$  is $r$-unsolvable (see Figure \ref{1}) and so 
	 	$ \pi(F_{n})=2n+2$. 
	 		\item[(ii)]
	 		Suppose that $V(K_n)=\{v_1,v_2,...,v_n\}$ and consider an arbitrary target vertex $r$, say $v_n$. Let $C$ be a configuration with $n+2$ pebbles such that $C(v_1)=3, C(v_2)=1, C(V_3)=2$ and $C(v_i)=1$ for $4\leq i\leq n-1$. 	 	
	 		In any case one can put two  pebbles on the target vertex. So $ \pi_{2}(K_{n})\leqslant n+2$.
	 			 		Now, consider $r$-unsolvable configuration with $n+1$ pebbles so that three  pebbles on a vertex (except target vertex $r$) and one pebble on other vertices. Then target vertex receives a pebble exactly, so $ \pi_{2}(K_{n})>n+1$. \qed
	 		\end{enumerate} 
	 
\begin{figure}[ht]
\centering
\includegraphics[scale=1.07]{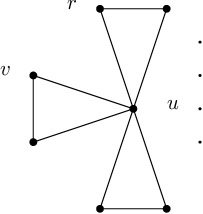}
\caption{The $r$-unsolvable configuration of the friendship graph. }\label{1}
\end{figure}

  \begin{figure}[ht]
  	\centering
  	\includegraphics[scale=1.07]{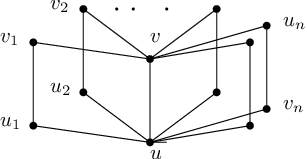}
  	\caption{The book graph $ B_{n}.$}\label{book}
  \end{figure}
  
 The $n$-book graph ($ n\geqslant 2 $) is defined as the Cartesian product $K_{1,n}\square P_{2}$. We call every $ C_{4} $ in the book graph $ B_{n} $, a page of $ B_{n} $. All pages in $ B_{n} $ have a common edge $ vu $ (see Figure \ref{book}). Here we compute the pebbling number of the book graph.

 \begin{figure}[ht]
 	\centering
 	\includegraphics[scale=1]{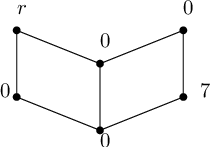}
 	\caption{The $r$-unsolvable  configuration of the graph $ B_{2} $.}\label{4}
 \end{figure}
 \begin{theorem}
 	\begin{enumerate} 
 		\item[(i)]
 			$ \pi(B_{2})=8. $
 		\item[(ii)] 
 		 	$\pi(B_{n})=2n+4$ for $ n\geqslant 2. $
\end{enumerate} 
 \end{theorem}
 \begin{proof}
 		\begin{enumerate} 
 			\item[(i)]  
 			Since $ C_{6} $ is a subgraph $ B_{2} $, so $ \pi(B_{2})\leqslant\pi(C_{6}) $. Now we consider $ r $-unsolvable configuration with size $ 7 $ in Figure \ref{4}. Therefore we have the result.
 			
 			\item[(ii)]  
 			 	We consider two cases:
 	\begin{enumerate} 
 		\item[(1)]
 		Let the vertex $ v $ (or $ u $) in Figure \ref{book} be a target vertex $ r $, so $ r $ is located on the each page. Since $\pi(C_{4})=4 $ and both of vertices $u$ and $ v $ are common for pages, we have a following $r$-unsolvable configuration of the book graph $ B_{n} $ with maximum size $ 2n+1 $ by either putting one pebble on any vertex or putting three pebbles on the vertex and one pebble on every vertex except its neighbours. 
 		\item[(2)]
 		Suppose that the target vertex is not the vertices $ u $ and $ v $. We consider $r$-unsolvable configuration for the subgraph $ B_{2} $ (Part (i)) and put one pebble on the rest of vertices. This is a maximal $r$-unsolvable configuration with $ 2(n-2)+7 $. Therefore we have the result. \qed
 	\end{enumerate}	
 	\end{enumerate}     
 \end{proof}

Pebbling on graph products and other
 binary graph constructions studied in \cite{Asplund}.  A center of interest of graph pebbling is Graham's conjecture (for the pebbling number of Cartesian product of two graphs), which is proposed by Chung \cite{Chung}  and states that
 $\pi(G\Box H)\leq \pi(G)\pi(H)$. A  result concerning with this conjecture is proposed in the paper of Asplund,
 Hurlbert and Kenter \cite{Asplund} states that  $\pi(G\Box H)\leq 2\pi(G)\pi(H)$. A recent result in 
 \cite{TOC} 
 improves this upper bound to $\big(2-\frac{1}{min\{\pi(G),\pi(H)\}}\big)\pi(G)\pi(H)$.

 The corona product $G\circ H$ of two graphs $G$ and $H$ is defined as the graph obtained by taking one copy of $G$ and $\vert V(G)\vert $ copies of $H$ and joining the $i$-th vertex of $G$ to every vertex in the $i$-th copy of $H$. Asplund, Hurlbert and Kenter \cite{Asplund} presented an upper bound for the pebbling number of coronas as follows:
 
 \begin{theorem} {\rm \cite{Asplund}}
 	For two graphs $G, H$ of orders $g$ and $h$, respectively, we have
 	\[
 	\pi(G\circ H)\leq 4 \pi(G)+gh. 
 	\]
 	\end{theorem}

 	Also, Pleanmani, Nupo and Worawiset in \cite{TOC} proposed an alternative bound which is related to the $2$-pebbling number:
 	
 	 \begin{theorem} {\rm \cite{TOC}}
 	 	For two graphs $G, H$ of orders $g$ and $h$, respectively, we have
 	 	\[
 	 	\pi(G\circ H)\leq 3 \pi_2(G)+gh. 
 	 	\]
 	 \end{theorem} 
 	 Now, we propose a formula for $\pi(G\circ  H)$  which is related to the $2$-pebbling number.

\begin{theorem} 
	If $G$ and $H$ are  two graphs of order $g$ and $h$, respectively. Then 
\begin{center}
$ \pi(G\circ  H)=gh+2(\pi_{2}(G)-1).$ 
\end{center} 
\end{theorem} 
\proof 
Let $r'$ be an arbitrary vertex in the  graph $G$. We need at least $\pi(G)$ pebbles for $r'$-solvable configuration. Every vertex of $G$ is connected to all vertices of a copy of $H$. If the vertices of the graph $H$ have at least two pebbles, then supporting pebbles enter the graph $G$, so the maximum number of pebbles that gives an $r'$-unsolvable configuration is equal to $ gh+2(\pi(G)-1)$. Now, if the target vertex $r$ is located in one of the copies of $H$, then the vertex $r $ must be $2$-solvable so that the vertex $r$ takes one pebble. So in this case, the maximum number of pebbles that have an $r$-unsolvable configuration is equal to $ (gh-1)+2(\pi_{2}(G)-1)$. Since $ \pi_{2}(G)>\pi(G) $, so
\begin{equation*}
\max {\lbrace gh+2(\pi(G)-1), (gh-1)+2(\pi_{2}(G)-1)\rbrace }=(gh-1)+2(\pi_{2}(G)-1)
\end{equation*}
which result is $ r $-solvable configuration by adding one pebble.\qed

\begin{corollary} 
For any graph $H$ of order $h$, $\pi(K_{n}\circ H)=nh+2n+2.$
\end{corollary}

The neighbourhood corona product $ G\star H $ of two graphs $ G $ and $ H $ is defined as the graph obtained by taking one copy of $ G $ and $\vert V(G)\vert $ copies of $H$ and joining the neighbours of the $i$-th vertex of $G$
to every vertex in the $i$-th copy of $ H $  (see e.g., \cite{neighber}).

Since the corona product $ G\circ H $ is an spanning subgraph of the neighbourhood corona product $ G\star H $, so we have the following proposition:

\begin{proposition}
For any two graphs $G$ and $H$,  $ \pi(G\star H)\leqslant\pi(G\circ H) $.
 \end{proposition} 
 In the following theorem, we show that the neighbourhood corona product $ K_{3}\star K_{m}$ is class 1. 
 \begin{theorem}
 	The neighbourhood corona $ K_{3}\star K_{m}$ is Class $ 1 $. 
 \end{theorem}
 \begin{proof}
 	Since $ {\rm diam}(K_{3}\star K_{m})=2$, so we consider the following $ r $-unsolvable configuration that the vertex $ r $ is a target vertex in $i$-th copy of $K_{m}$. Put three  pebbles on one of the vertices of any  copy of $ K_{m} $ (except $i$-th copy) and put one  pebble on any other vertex except vertices of $ K_{3}$. We will have an $ r $-unsolvable configuration with $ 3m+3 $ pebbles (see Figure \ref{2}). So we have the result.  
 \end{proof}
 
 \begin{figure}[ht]
 	\centering
 	\includegraphics[scale=1.01]{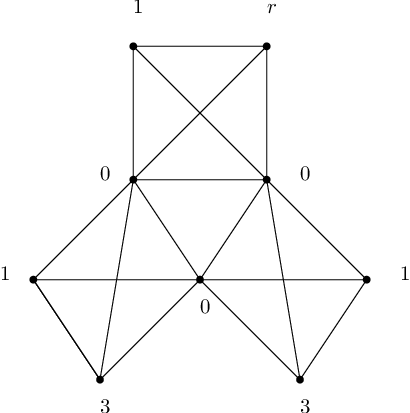}
 	\caption{\label{2} The $r$-unsolvable configuration of the  graph $ K_{3}\star K_{2} $.} 
 \end{figure}
 
 	Since $ K_{3}\star H $ is a spanning subgraph with diameter two of $ K_{3}\star K_{m}$ (where $ H $ is spanning subgraph of $K_{m}$), so we have the following corollary. 
 \begin{corollary} 
 	$K_{3}\star H $ is class 1.
  \end{corollary}

  \begin{theorem}
  	The neighbourhood corona $ K_{n}\star G$ is Class $ 0 $ for $ n\geqslant 4$. 
  \end{theorem}
  \begin{proof}
  	Let $C$ be an arbitrary configuration of the graph $ K_{n}\star G $ with size of $ \vert V(K_{n}\star G)\vert $ and, $ r $ be the target vertex. By pigeonhole principle, there is a vertex of graph such that receives more than one pebble. Since $ diam(K_{n}\star G)=2$, if at least $ 4 $ pebbles put on a vertex of graph, then target vertex $ r $ will be solvable. Considering all three arbitrary vertices of the graph have a common neighbour of $ K_{n} $,  so two distinct vertices that each others have at least $ 2 $ pebbles can put one pebble on the target vertex. The last case is that one vertex has three pebbles and the other vertices have at most one pebble, in this case only two vertices will be without pebbles, so $ r $ will be solvable.
  \end{proof}
  \begin{corollary}
  	$ K_{n}\star K_{m} $ for $ n\geqslant 4 $ is class $0$.
  \end{corollary}

  \section{Optimal pebbling number}  
  In this section, we consider the best scenario, that is, the smallest configuration that can solve any goal. The optimal pebbling number of a graph $ G $, denoted by the symbol $ \pi^{*}(G) $, is equal to the smallest number $ m $ such that there is a  configuration $ C $ with size  $ m $ which is a solution for each target $ r $.   
  
  \begin{theorem}
  	For any graph $H$ and any complete graph $ K_{n}$, 
  	\begin{itemize}
  		\item[i)]
  		$ \pi^{*}(K_{n}\circ H)=4 $ for $ n\geqslant 3$,
  		\item[ii)]
  		$ \pi^{*}(K_{n}\star H)=3 $ for $ n\geqslant 2$.
  	\end{itemize}
  \end{theorem}
  \begin{proof}
  	\begin{itemize}
  		\item[i)]
  		Note that $ \pi^{*}(K_{n})=2 $, therefore one can reach $ 2 $ pebbles to any vertex by putting  $ 2 $ pebbles on the each of two vertices $ v $ and $ w $ of graph $ K_{n} $. Therefore the result follows.
  		\item[ii)]
  		Since any vertex $ v $ of the graph $ K_{n} $ is adjacent to $ n-1 $ copies of $ H $, so, it is satisfied to put $ 2 $ pebbles on the vertex $ v $ and put $ 1 $ pebble on another vertex $ w $ of $ K_{n} $.
  	\end{itemize}
  \end{proof}
  
  \section{Conclusion}
  In this paper we studied the pebbling number of corona and neighberhood corona  of two graphs. We obtained a result for the optimal pebbling number of corona and neighberhood corona of $ K_{n} $ and any graph. Also there many problems in the cover pebbling number of graphs such as cycles.  We close the paper with the following problems: 
  
  \medskip
  \noindent {\bf Problem 1}: What can be say about the pebbling number of the neighberhood corona of two arbitrary graphs?

%----------------------------------------------------------------------
%----------------------------------------------------------------------
%----------------------------------------------------------------------

\end{document}